\documentclass[11pt]{article}
\usepackage[margin=1in]{geometry}
\usepackage{amsmath,amssymb,amsthm,mathtools}
\usepackage{microtype}
\usepackage{enumitem}
\usepackage{booktabs}
\usepackage[hidelinks]{hyperref}
\usepackage[T1]{fontenc}
\usepackage{lmodern}

\newtheorem{principle}{Principle}[section]
\newtheorem{axiom}{Axiom}[section]
\newtheorem{definition}{Definition}[section]
\newtheorem{proposition}{Proposition}[section]
\newtheorem{remark}{Remark}[section]

\newcommand{\Qp}{\mathbb{Q}_p}
\newcommand{\Zp}{\mathbb{Z}_p}
\newcommand{\Pcal}{\mathcal{P}}
\newcommand{\Bcal}{\mathcal{B}}
\newcommand{\Tcal}{\mathcal{T}}
\newcommand{\Lcal}{\mathcal{L}}
\newcommand{\Ical}{\mathfrak{I}}
\newcommand{\depth}{\operatorname{depth}}
\newcommand{\rad}{\operatorname{rad}}
\newcommand{\Child}{\operatorname{Child}}

\title{\textbf{Beyond Archimedean Intelligence:\\
Toward an Intrinsic \(p\)-Adic Theory of Learning}}
\author{Bourama Toni\\Department of Mathematics. Howard University}
\date{}

\begin{document}
\maketitle

\begin{abstract}
Overwhelmingly Archimedean/Euclidean, the mathematical foundations of contemporary machine learning involve the following: data are typically embedded in real vector spaces, similarity is measured by Euclidean or related norms, learning is formulated through real-valued loss functions, and differential or gradient-based optimization generally drives adaptation. Recent investigations of \(p\)-adic and ultrametric methods in machine learning have demonstrated the potential of non-Archimedean structures for hierarchical representation, neural architectures, classification, and information processing.

Our goal is not merely to construct a \(p\)-adic analogue of an existing Euclidean neural network, rather at initiating the study of learning systems whose primitive mathematical structure is intrinsically non-Archimedean.

Therefore, this note, not a manifesto,  advocates a more foundational point of view. Indeed, rather than searching for \(p\)-adic realizations, approximations, or footprints within existing Archimedean learning systems, we propose that a genuinely non-Archimedean theory of intelligence should be derived intrinsically from the topology, geometry, algebra, and measure structure of \(p\)-adic spaces. In particular, neurons, layers, activation functions, gradients, and backpropagation should not be assumed \emph{a priori} to be primitive concepts of learning.

We formulate a \emph{Non-Archimedean Autonomy Principle} and introduce an axiomatic framework in which information states are organized by ultrametric neighborhoods, learning is represented by hierarchical refinement and redistribution of information across nested balls, and scale is determined by valuation depth. We show a simple structural result showing that, on \(\Zp\), refinement depth is exactly common-prefix depth and that the finite truncations \(\mathbb Z/p^N\mathbb Z\) preserve the ball hierarchy through level \(N\). Thus hierarchy is not an auxiliary representation to be learned: it is already encoded in the geometry.

\end{abstract}



\section{Introduction}

\emph{Information is Archimedean}. Or is it not? This is the implicit geometric assumption powering modern artificial intelligence and machine learning, and one should acknowledge its remarkable success, akin to the success enjoyed by Physics for over a century. The mathematical foundations in both cases are Archimedean/Euclidean, upholding the Archimedean principle of \emph{There are arbitrary "big" integers}. 

That is, a state, physical or data space underlying any learning system starts on some space
\[
X\subseteq\mathbb R^n,
\]
equipped with Euclidean geometry and its "as the crow flies" flatness principle. A norm \(||x-y||\) or inner product \( \langle x,y\rangle\), is used to express similarity. Learning is frequently formulated as minimization of a real-valued functional
\[
L:\Theta\longrightarrow\mathbb R.
\]
A local differential mechanism such as
\[
\theta_{k+1}=\theta_k-\eta\nabla L(\theta_k).
\]
governs the evolution of parameters.

However, despite the extraordinary achievements of neural networks and modern machine learning and artificial intelligence there is no established claim that \emph{the Archimedean framework is mathematically necessary for intelligence or learning.} It is only a powerful testament that intelligence-like computational behavior can be modeled successfully along Archimedean principles. Thus  our central question:
\[
\text{Should a mathematical theory of learning be Archimedean?}
\]
or
\[ 
\text{Must Intelligence itself therefore be Archimedean?}
\]
\[
\text{If not, then what viable alternative framework is equally mathematically rigorous and computationally efficient? }
\]
Here we advocate a Non-Archimedean or ultrametric framework using the p-adic space as reference.

For \(p\) prime, the field of p-adic numbers,  \(\Qp\), is the completion of \(\mathbb Q\) with respect to the p-adic norm or ultra norm
\[
|x|_p=p^{-v_p(x)},
\]
with $v_p(x)$ the so-called  p-valuation of $x,$ expressed as, for nonzero rational $x=p^{v_p}\frac{a}{b},$ $a,b$ co-prime with $p,$ . $v_p(0):=\infty.$ The induced metric or ultrametric
\[
d_p(x,y)=|x-y|_p
\]
satisfies the strong triangular inequality
\[
d_p(x,z)\leq\max\{d_p(x,y),d_p(y,z)\}.
\]
Thus \(\Qp\) is ultrametric. The subset $\mathbb{Z}_p=\{x\in \mathbb{Q}_P:\quad |x|_p \le 1 \}$ is the ring  of p-adic integers. We recall

\[
\mathbb{Z}_p \ni x= \sum_{i=0}^\infty x_i p^i:=\cdots x_2x_1x_0x_{-1}\cdots x_{-v_p} ,\quad x_i \in \{0,\dots, p-1\},
\]
also called a \emph{string over an alphabet of p symbols}. p-adic expansion is expressed uniquely as 
\[
\mathbb{Q}_p \ni x = \sum_{i=-v_p}^\infty x_ip^i:=\cdots x_2x_1x_0x_{-1}\cdots x_{-v_p},\quad x_i \in \{0,\dots, p-1\}.
\]
where \(v_p \in \Zp\), \(x_i \in \{0,\ldots,p-1\}\), and 
\( x_{-v_p} \ne 0\).
Additionally,
\[ 
\mathbb Q \subset \Qp,\quad \mathbb Z \subset \Zp,\quad \Zp \cap \mathbb Q = \{\frac{a}{b}:\quad p \nmid b\}
\]
The ensuing geometry has some peculiar properties we summarize as follows: 

\begin{itemize}
\item Closed and open balls respectively denoted
\[
B_r[a]=\{x\in\Qp: |x-a|_p\le r\}.\quad B_s(b)=\{x\in\Qp : |x-b|_p < s\}
\]
are simultaneously open and closed, including the so-called \emph{clopen topology} on \(\Qp\).
\item Balls are nested or disjoint; \( \forall r\le s\in \mathbb R_{+},\quad \forall a,b\in\Qp\), either \(B_r(a)\cap B_s(b) =\empty\) or \( B_r(a) \subset B_s(b) \).
\item Every point of a ball is also its center; \(B_x(a)=B_r(x),\quad \forall x\in B_r(a)\)
\item Intersecting balls of equal radius coincide
\item All triangles are isosceles
\item Scale is naturally encoded by valuation. 
\item Importantly, 
\[
\Zp=B_1[0]= B_{1/p}[0] \sqcup B_{1/p}[1] \sqcup \ldots \sqcup B_{1/p}[p-1].
\]
 that is, a decomposition into a disjoint union of p smaller balls. Thus the hierarchical structure of \(\Zp\). 
 \item  Likewise, 
 \[
 B_{p^n}[a] = \bigsqcup _{i=0}^{p-1} B_{p^{n-1}}[a],\quad \forall B_{p^n}[a] \subset \Qp 
 \]
\end{itemize} 

Consequently, hierarchy is not an external structure imposed upon \(p\)-adic information:

\[
\textbf{Hierarchy in essence is already geometry.}
\]

Some recent work has shown that non-Archimedean mathematics can play a substantive role in learning theory through \(p\)-adic neural networks and cellular neural networks, \(p\)-adic statistical field theories associated with deep belief networks, and \(p\)-adic approaches to representation learning \cite{AT,Martins,Mihara,Nguessan,ZZG,Zuniga,Torresblanca}. 

Here we consider a much different logical starting point. That is, we do not ask questions such as 
\[
\text{What is the \(p\)-adic analogue of a neural network?}
\]
Instead, we are more interested in questions such
\[
\text{What is learning if the information universe is \(p\)-adic from the beginning?}
\]

\section{Relation to Existing p-adic Machine Learning and AI systems}

Existing work on p-adic machine learning such as the referenced papers has demonstrated that non-Archimedean mathematics can successfully enrich neural architectures, representation learning, classification, and information processing. In nearly all such approaches, however, the starting point remains an existing learning architecture whose components are subsequently reformulated over p-adic or ultrametric spaces.

The viewpoint adopted here is fundamentally different. Rather than asking how existing learning systems may be implemented over p-adic domains, we ask what mathematical structures should be regarded as \emph{primitive} if the information universe were non-Archimedean from the outset.

Consequently, neurons, layers, activation functions, gradients, and backpropagation are treated not as axioms but as possible emergent computational realizations of a deeper intrinsic learning theory.

However distinct, the two viewpoints are complementary. Existing p-adic machine learning/artificial intelligence investigates non-Archimedean realizations of contemporary learning systems, whereas the present work investigates the intrinsic mathematical foundations from which non-Archimedean learning itself might emerge.

\section{Principle of Intrinsic Non-Archimedean Learning}

\begin{principle}[Intrinsic Non-Archimedean Learning]

A theory of learning over a non-Archimedean information space should be derived from the intrinsic topology, geometry, algebra, and measure structure of its ultrametric state space. Any Archimedean interpretation and realization, Euclidean embedding, or computational implementation is secondary to the intrinsic mathematical theory
\end{principle}

The principle describes an order of construction:
\[
\text{non-Archimedean structure}
\to
\text{learning primitives}
\to
\text{learning dynamics}
\to
\text{architecture}
\to
\text{physical realization}.
\]

However, the reverse procedure,
\[
\text{existing neural architecture}\to\text{\(p\)-adic modification},
\]
may yield useful computational models, but it need not reveal the intrinsic mathematical form of non-Archimedean learning. In particular, the  non-archimedean autonomy requires that the concepts
\[
\text{neuron},\quad \text{layer},\quad \text{activation},\quad
\text{gradient},\quad \text{backpropagation}
\]
are not assumed. They are successful primitives of Archimedean machine learning; they are not and should not be the axioms of intelligence.

\section{Ultrametric Information Spaces}

A classic/Archimedean information space is framed around: structure and semantics - navigability - codification and abstraction - multidimensionality - rule based dynamics - user centric architecture.

Let \((\mathbb X,d)\) be a complete ultrametric space, with the usual examples \(\mathbb X=\Zp^n\) and \(\mathbb X=\Qp^n\). For \(x\in \mathbb X\) and \(r>0\), write
\[
B_r(x)=\{y\in \mathbb X:d(x,y)\le r\},
\]
and let \(\Bcal(\mathbb X)\) denote the family of admissible information balls.

We interpret \(B\in\Bcal(\mathbb X)\) as an information class at a specified resolution. If \(B'\subset B\), then \(B'\) represents a refinement of the information encoded by \(B\).

For \(\mathbb X=\Zp\),
\[
\Zp\supset a_0+p\Zp\supset a_0+a_1p+p^2\Zp\supset\cdots.
\]
For
\[
x=a_0+a_1p+a_2p^2+\cdots,
\]
each additional \(p\)-adic digit increases information resolution.

\begin{definition}[Information depth]
If \(\rad(B)=p^{-k}\), the radius of an admissible information ball,  define
\[
\depth(B)=k
\]
as the information depth. 
\end{definition}
\begin{remark}
   Information refinement corresponds to increasing valuation depth. 
\end{remark}
\subsection{Intrinsic \(p\)-Adic Learning}

Let a complete ultrametric space \( (\mathbb X,d)) \) be given. We propose the following axiomatics for an intrinsic p-adic learning 

\begin{axiom}[Ultrametric state]
The primitive information state belongs to a complete ultrametric space \((\mathbb X,d)\). Similarity is determined intrinsically by \(d\), not by an auxiliary Euclidean embedding.
\end{axiom}

\begin{axiom}[Hierarchical locality]
Two information states interact according to their smallest common ultrametric neighborhood. When it exists, set
\[
B(x,y)=\min\{B\in\Bcal(\mathbb X):x,y\in B\},
\qquad
\ell(x,y)=\depth(B(x,y)).
\]
Locality is hierarchical rather than spatial.
\end{axiom}

\begin{axiom}[Refinement]
Learning permits transitions
\[
B\longrightarrow B_i,\qquad B_i\subseteq B.
\]
Thus learning changes the resolution at which information is distinguished.
\end{axiom}

\begin{axiom}[Measure redistribution]
An information state may be represented by \(\mu\in\Pcal(\mathbb X)\). Learning is an evolution
\[
\mu_t\longrightarrow\mu_{t+1},
\qquad
\mu_{t+1}=\Lcal_{\Theta_t}\mu_t,
\]
where \(\Lcal_{\Theta_t}:\Pcal(\mathbb X)\to\Pcal(\mathbb X)\) acts on the hierarchical information structure.
\end{axiom}

\begin{axiom}[Valuation scale]
For \(x,y\in\Qp\), the quantity \(v_p(x-y)\) measures the depth of common \(p\)-adic information. Similarity, resolution, and interaction scale arise from valuation.
\end{axiom}

\begin{axiom}[Non-Archimedean closure]
The fundamental learning dynamics are definable within the non-Archimedean information space before any Archimedean observation or implementation map is introduced. A map
\[
\Phi:\mathbb X\to \mathbb Y\subseteq\mathbb R^m
\]
is an observation, representation, or implementation map; it is not part of the intrinsic definition of learning.
\end{axiom}

\subsection{The Prefix--Ball Correspondence}

Our proposed viewpoint involves the following proposition readily available in many books on p-adic analysis. We include here a formulation and proof adapted to our framework

\begin{proposition}[Intrinsic hierarchy and finite preservation]\label{prop:prefix}
Let
\[
x=\sum_{j=0}^{\infty}a_jp^j,\qquad
y=\sum_{j=0}^{\infty}b_jp^j
\]
belong to \(\Zp\), with \(a_j,b_j\in\{0,\ldots,p-1\}\). Then, for \(k\ge1\),
\[
x\equiv y\pmod{p^k}
\quad\Longleftrightarrow\quad
a_j=b_j\ \text{for }0\le j<k
\quad\Longleftrightarrow\quad
|x-y|_p\le p^{-k}.
\]
Consequently, the depth of the smallest common ball containing \(x\) and \(y\) is determined exactly by their maximal common initial \(p\)-adic digit string. Moreover, for every \(N\ge1\), the truncation
\[
\pi_N:\Zp\to\mathbb Z/p^N\mathbb Z
\]
preserves all ball inclusions and separations through depth \(N\).
\end{proposition}

\begin{proof}
The congruence \(x\equiv y\pmod{p^k}\) is equivalent to \(p^k\mid(x-y)\), hence to
\[
v_p(x-y)\ge k.
\]
By uniqueness of the \(p\)-adic expansion, this occurs exactly when
\[
a_0=b_0,\ldots,a_{k-1}=b_{k-1}.
\]
Since \(|x-y|_p=p^{-v_p(x-y)}\), the valuation inequality is equivalent to
\[
|x-y|_p\le p^{-k}.
\]
Thus common-prefix depth, valuation depth, and ultrametric proximity coincide.
\end{proof}

\begin{remark}
Consequently, the depth of the smallest common ball containing \(x\) and \(y\) is determined exactly by their maximal common initial \(p\)-adic digit string. Moreover, for every \(N\ge1\), the truncation
\[
\pi_N:\Zp\to\mathbb Z/p^N\mathbb Z
\]
preserves all ball inclusions and separations through depth \(N\).
\end{remark}

Indeed, let \(B=a+p^k\Zp\) with \(k\le N\). Its image under \(\pi_N\) is the congruence class
\[
\pi_N(a)+p^k(\mathbb Z/p^N\mathbb Z).
\]
If \(B'\subseteq B\) has depth at most \(N\), the corresponding congruence conditions are nested and remain nested after reduction modulo \(p^N\). If two balls of depth at most \(N\) are disjoint, their defining prefixes differ in one of the first \(N\) digits, so their images under \(\pi_N\) remain disjoint. Hence the rooted ball hierarchy is preserved through level \(N\).

\begin{remark}
This Proposition \ref{prop:prefix} gives a precise form to the statement that hierarchy is already geometry. No auxiliary clustering procedure is required to manufacture the rooted hierarchy: the hierarchy is canonically encoded by congruence, valuation, and ultrametric balls. 

It also shows that finite \(p\)-adic computation is not merely an arbitrary approximation. The quotient \(\mathbb Z/p^N\mathbb Z\) retains exactly the first \(N\) levels of intrinsic information geometry.
\end{remark}

\section{Learning as Evolution on the Tree of Balls}

Let \(\Tcal_p\) denote the rooted \(p\)-ary tree of balls of \(\Zp\). A vertex at depth \(k\) represents
\[
B_{k,a}=a+p^k\Zp.
\]
As such each vertex has \(p\) \emph{children} (successors),
\[
B_{k+1,a+jp^k},\qquad j=0,\ldots,p-1.
\]

We propose that the primitive learning object need not be a neuron. It may instead be a \emph{hierarchical transition} rule
\[
T_\Theta(B,B')
\]
for \(B'\in\Child(B)\), satisfying
\[
T_\Theta(B,B')\ge0,
\qquad
\sum_{B'\in\Child(B)}T_\Theta(B,B')=1.
\]
The quantity \(T_\Theta(B,B')\) weights refinement from \(B\) toward \(B'\). A learning state is therefore a weighted hierarchical system
\[
(\Tcal_p,\mu_t,T_{\Theta_t}),
\]
and learning occurs precisely when
\[
\Theta_t\longrightarrow\Theta_{t+1},
\qquad
\mu_{t+1}=\Lcal_{\Theta_t}\mu_t.
\]

Now we present a more  explicit definition of what should be understood by \emph{Intrinsic Learning}.

\section{Intrinsic Learning: A Definition}

\begin{definition}[Intrinsic \(p\)-adic learning system]
An intrinsic \(p\)-adic learning system is a quadruple
\[
\mathfrak L_p=(\mathbb X,\Bcal,\mu,\Lcal),
\]
where:
\begin{enumerate}[label=(\roman*)]
\item \(\mathbb X\) is a complete \(p\)-adic or ultrametric information space;
\item \(\Bcal\) is its hierarchical family of information balls;
\item \(\mu\in\Pcal(\mathbb X)\) is an information-state measure;
\item \(\Lcal\) is a family of hierarchical information-evolution operators acting intrinsically on \(X\).
\end{enumerate}
\end{definition}

An information-state measure is also referred to as \emph{entropy} and is used to quantity uncertainty. 

Of note, a \emph{learning trajectory}  is a sequence \((\mu_t)_{t\ge0}\) satisfying
\[
\mu_{t+1}=\Lcal_{\Theta_t}\mu_t.
\]

The system \(\mathfrak L_p\) learns relative to an information criterion \(\Ical\) if
\[
\Ical(\mu_{t+1})>\Ical(\mu_t)
\]
for an appropriate class of learning transitions. A functional
\[
I: \Pcal(\mathbb X) \longrightarrow \mathbb R
\]
is intrinsic if it depends only upon the ultrametric structure and is invariant under ultrametric isometries.

The criterion \(\Ical\) may be related to discrimination information/generalized discrimination information, predictive information, hierarchical entropy, or a non-Archimedean model-improvement functional. This raises a problem central to our viewpoint:
\[
\text{What is the intrinsic information criterion of \(p\)-adic learning?}
\]

And should Hierarchy be learned?

\section{Hierarchy Need Not Be Learned}

In a conventional deep network, hierarchy is often represented schematically as
\[
\text{input}\to\text{features}\to\text{higher features}\to\text{representation}.
\]
In an ultrametric information space,
\[
B_0\supset B_1\supset B_2\supset\cdots
\]
already provides a hierarchy of resolutions.

Consequently, p-adic learning may not require the construction of hierarchy. It may consist of discovering how information should move through an existing hierarchy: In other words

\[
\text{Archimedean learning often constructs hierarchy;}
\]
However, 
\[
\text{\(p\)-adic learning may just move through hierarchy. Depth in a non-Archimedean learning system may be identified with valuation depth rather than network depth}
\]
That is
\[
\textbf{Hierarchy is not Learned. Hierarchy is Geometry}
\]

If hierarchy is already encoded in the geometry, then what should be guiding the set of computational tasks?

\section{From Intrinsic Theory to Finite Computation}

For
\[
x=x_0+x_1p+\cdots+x_{N-1}p^{N-1}+\cdots,
\]
define
\[
\pi_N(x)=x_0+x_1p+\cdots+x_{N-1}p^{N-1}.
\]
The finite information space is
\[
\mathbb{X}_N=\mathbb Z/p^N\mathbb Z.
\]
By Proposition \ref{prop:prefix}, the truncated ball tree retains the intrinsic hierarchy through depth \(N\). Thus one obtains finite systems
\[
\mathfrak L_p^{(N)}
=
(\mathbb{X}_N,\Bcal_N,\mu_N,\Lcal_N).
\]

The above analysis would strongly suggest the three steps
\begin{enumerate}
\item Intrinsic \(p\)-adic learning theory
\item Finite hierarchical computation modulo \(p^N\)
\item Hardware realization
\end{enumerate}

We emphasize that the mathematical structure should determine the computational primitives such as :
digit comparison, valuation detection; prefix agreement; tree routing; hierarchical aggregation. However, It is customary in the current state of development of machine learning and artificial intelligence that hardware realization requirements precede and often dictate the mathematics, e.g., constructing a p-adic analogue of existing neural networks/microchips.

The hardware question is therefore not merely how existing processors can imitate \(p\)-adic arithmetic. 

\[
\text{What hardware primitives are demanded by intrinsic \(p\)-adic learning?}
\]

Our viewpoint therefore prompts the following mathematical problems and research directions we are currently and diligently addressing.

\section{Open Mathematical Problems}

\begin{enumerate}
\item Characterize information functionals
\[
\Ical:\Pcal(X)\to\mathbb R
\]
intrinsic to ultrametric refinement.

\item Define a non-Archimedean model-improvement information associated with
\[
\mu_t\to\mu_{t+1}.
\]

\item Determine conditions under which \(\Lcal_\Theta\) admits fixed points, periodic states, or invariant measures.

\item Develop a stability theory for learning trajectories on \(\Pcal(\Zp)\).

\item Determine whether intrinsic learning dynamics should use discrete time, real time, or \(p\)-adic time.

\item Formulate attention intrinsically through valuation depth and common ultrametric neighborhoods.

\item Characterize generalization as stability of learned behavior on ultrametric balls.

\item Construct finite learning systems over \(\mathbb Z/p^N\mathbb Z\) converging, in an appropriate sense, to an intrinsic \(p\)-adic learning system.

\item Identify minimal computational primitives for a native hierarchical \(p\)-adic processor/microchip.

\item Determine whether a theory of \(p\)-adic learning requires objects analogous to neurons at all.
\end{enumerate}

\section{Concluding Remarks}

The mathematical foundations of contemporary machine learning are predominantly Archimedean. This reflects the historical development and extraordinary success of real-valued analysis, optimization, and computation. It should not be interpreted as a mathematical necessity.

The \(p\)-adic numbers provide a complete non-Archimedean information space in which hierarchy, scale, and locality have fundamentally different meanings. If learning is formulated intrinsically in such a space, familiar concepts such as neurons, layers, activation functions, gradients, and backpropagation may cease to be primitive.

We have proposed the fundamental \emph{Non-Archimedean Autonomy Principle}:

\[
\text{Derive learning from intrinsic geometry before imposing physical realizability.}
\]

Within this viewpoint, information is organized by nested ultrametric balls, scale is valuation depth, and learning is represented by refinement and redistribution across a hierarchical state space. The aboce Proposition \ref{prop:prefix} shows that common-prefix depth, valuation depth, and ultrametric proximity coincide on \(\Zp\), while finite quotients preserve the hierarchy through their truncation depth.

The core claim here is:
\[
\textbf{The mathematical theory of intelligence should not be assumed a priori to be Archimedean.}
\]

The objective is not to \(p\)-adicize existing artificial intelligence/artificial general intelligence. It is to ask what intelligence might mean in a mathematical universe that was non-Archimedean from the beginning: Determine the mathematical principles from which an intrinsically non-Archimedean theory of intelligence would emerge.

In an intrinsic non-Archimedean theory of intelligence, hierarchy is not an architecture constructed and imposed upon information; it is an inherent property of the information space itself. That is, 
\[
\textbf{Hierarchy is not learned. It is geometry.}
\]

\end{document}